\documentclass[10pt]{amsart}
\usepackage{fancyhdr}
\usepackage{stmaryrd}
\usepackage{calrsfs}

\usepackage{amsmath}
\usepackage[nospace,noadjust]{cite}
\usepackage{amsfonts}
\usepackage{amssymb,enumerate}
\usepackage{amsthm}
\usepackage{cite}
\usepackage{comment}
\usepackage{color}
\usepackage[all]{xy}
\usepackage{hyperref}
\usepackage{lineno}
\usepackage{stmaryrd}

\usepackage{mathtools}
\usepackage{bm}
\usepackage{graphicx}
\usepackage{psfrag}
\usepackage{url}

\usepackage{fancyhdr}
\usepackage{tikz-cd}

\usepackage{dirtytalk}
\usepackage{float}
\usepackage{mathrsfs}

\newtheorem*{maintheorem*}{Main Theorem}
\newtheorem{theorem}{Theorem}[section]

\newtheorem{lemma}[theorem]{Lemma}
\newtheorem{cor}[theorem]{Corollary}

\theoremstyle{definition}
\newtheorem{definition}[theorem]{Definition}

\newtheorem{example}[theorem]{Example}
\numberwithin{equation}{section}

\newcommand{\supp}{\mathsf{supp}}

\hypersetup{
	pdftitle={FD},
	colorlinks=true,
	linkcolor=blue,
	citecolor=cyan,
	urlcolor=wine
}



\begin{document}
	
	\mbox{}
	\title{A Goldbach Theorem for Laurent Polynomials with Positive Integer Coefficients}
	
	\author{Sophia Liao}
	\address{Proof School, 973 Mission St., San Francisco, CA 94103}
	\email{sophialiao@proofschool.org}
	
	\author{Harold Polo}
	\address{Department of Mathematics\\University of California, Irvine\\Irvine, CA 92697}
	\email{harold.polo@uci.edu}
	
\date{\today}

\begin{abstract}
	We establish an analogue of the Goldbach conjecture for Laurent polynomials with positive integer coefficients.
\end{abstract}
\medskip

\maketitle


\bigskip
\section{Introduction}
\label{sec:intro}

The celebrated Goldbach conjecture has stood as one of the most intriguing unsolved problems in mathematics. In modern terms, the conjecture asserts that every even integer strictly greater than $2$ can be written as the sum of two prime numbers. Over the years, many mathematicians have devoted their efforts to exploring the intricacies of the Goldbach conjecture~\cite{chen, helfgott, vinogradov}. 

The scope of investigations surrounding the Goldbach conjecture has expanded beyond its original formulation, inspiring researchers to explore analogous ideas for classes of polynomial rings~\cite{hayes,pollack,saidak}. Hayes~\cite{hayes} found representations for every polynomial $f \in \mathbb{Z}[x]$ with $\deg(f) \geq 1$ as the sum of two irreducible polynomials. Building upon this, Pollack~\cite{pollack} extended the result to Noetherian domains with infinitely many maximal ideals, while Saidak~\cite{saidak} studied the number of representations of a polynomial with integer coefficients as the sum of two irreducibles. In the context of polynomials over finite fields of characteristic $2$, Effinger and Hayes~\cite{effinger} introduced a notion of parity and demonstrated that most odd monic polynomials can be expressed as the sum of three irreducibles.

In what follows, we provide an analogue of the Goldbach conjecture for Laurent polynomials with positive integer coefficients. First, we introduce a couple of key definitions. Throughout this paper, we let $\mathbb{N}_0[x^{\pm 1}]$ represent the set consisting of all Laurent polynomials with positive integer coefficients. On the other hand, given a polynomial $f = \sum_{i = 0}^{n} c_i x^{k_i} \in \mathbb{N}_0[x^{\pm 1}]$, we refer to the set $\{k_i \mid i \in \llbracket 0,n \rrbracket\}$ as the \emph{support} of $f$, which we denote by $\supp(f)$.  The arithmetic properties of $\mathbb{N}_0[x^{\pm 1}]$ (and other classes of polynomial semirings) have generated much interest lately~\cite{CF19,chapmanpolo, gotti}. We are now ready to introduce a counterpart of the Goldbach conjecture for polynomials in $\mathbb{N}_0[x^{\pm 1}]$.

\smallskip
\begin{theorem} \label{theorem: our main result}
	Every polynomial $f \in \mathbb{N}_0[x^{\pm 1}]$ can be written as the sum of two irreducibles provided that $f(1) > 3$ and $|\supp(f)| > 1$.
\end{theorem}
\smallskip

The statement of Theorem~\ref{theorem: our main result} is intimately related to the Goldbach conjecture. Observe that, since $\mathbb{N}_0[x^{\pm 1}]^{\times} = \{x^k \mid k \in \mathbb{Z}\}$, we can think of a polynomial in $\mathbb{N}_0[x^{\pm 1}]$ as a formal sum of (multiplicative) units. This implies that, as in the Goldbach conjecture, proving Theorem~\ref{theorem: our main result} involves partitioning a fixed set of units into two subsets, each one of which represents an irreducible Laurent polynomial with positive integer coefficients. Furthermore, if the Goldbach conjecture were true then, using the fact that $f \in \mathbb{N}_0[x^{\pm 1}]$ is irreducible when $f(1)$ is a prime number (see Lemma~\ref{lemma: f(1) is prime} below), we could easily show that Theorem~\ref{theorem: our main result} holds for polynomials $f \in \mathbb{N}_0[x^{\pm 1}]$ satisfying that $f(1)$ is an even number strictly greater than $2$. 

But why do we not consider a statement similar to Theorem~\ref{theorem: our main result} in the context of $\mathbb{N}_0[x]$? It is worth mentioning that such a statement does not hold as, for any $k \geq 2$, the polynomial $x^{k} + \cdots + x$ cannot be expressed as the sum of two irreducibles.

\section{Proof of Theorem~\ref{theorem: our main result}.}

Hold your horses! Before we jump into the excitement of proving Theorem~\ref{theorem: our main result}, let us handle the not-so-thrilling task of setting up some notation. Throughout the proof, we let $\mathbb{Z}, \mathbb{N}$, $\mathbb{N}_0$, and $\mathbb{P}$ denote the set of integers, positive integers, nonnegative integers, and prime numbers, respectively. Moreover, for $m,n \in \mathbb{N}_0$, we set
\[
\llbracket m,n \rrbracket \coloneqq \left\{k \in \mathbb{N}_0 \mid m \leq k \leq n\right\}.
\]
Observe that if $m > n$, then $\llbracket m,n \rrbracket = \emptyset$. To avoid pedantic repetitions, whenever we consider a polynomial $\sum_{i = 0}^{n} c_i x^{k_i}$ in $\mathbb{N}_0[x^{\pm 1}]$, we assume that $k_0 > \cdots > k_n$ and $c_0, \ldots, c_n \in \mathbb{N}$, unless we specify otherwise. Now, giddy up!

Our first goal is to identify enough irreducibles within $\mathbb{N}_0[x^{\pm 1}]$. Hayes~\cite{hayes} proved a Goldbach conjecture analogue for polynomials in $\mathbb{Z}[x]$ by cleverly manipulating the coefficients of a polynomial $f \in \mathbb{Z}[x]$. Therefore, examining the coefficients of a polynomial appears to be a natural starting point.

\begin{lemma} \label{lemma: f(1) is prime}
	Let $f \in \mathbb{N}_0[x^{\pm 1}]$ such that $f(1)$ is a prime number. Then $f$ is irreducible.
\end{lemma}

\begin{proof}
	Write $f = gh$ with $g,h \in \mathbb{N}_0[x^{\pm 1}]$. Since $f(1) = g(1)h(1)$ and $f(1)$ is a prime number, we may assume that $g(1) = 1$. Hence the element $g$ is a unit of $\mathbb{N}_0[x^{\pm 1}]$. We can then conclude that $f$ is irreducible.
\end{proof} 

\begin{cor} \label{cor: goldbach for numbers}
	Let $f \in \mathbb{N}_0[x^{\pm 1}]$ such that $f(1) = p + q$, where $p$ and $q$ are prime numbers. Then $f$ can be written as the sum of two irreducibles of $\mathbb{N}_0[x^{\pm 1}]$.
\end{cor}

Simply examining the coefficients of a polynomial does not seem to be enough to tackle Theorem~\ref{theorem: our main result}. In fact, if we want to write polynomials of the form $\sum_{i = 0}^n x^{k_i}$ as the sum of two irreducibles, then we need to come up with irreducibility criteria that consider the exponents of a polynomial rather than the coefficients. This is precisely the purpose of the following lemmas and definitions up until Lemma~\ref{lemma: splitting configuration}. 

\begin{definition}
	Following~\cite{chapmanpolo}, we say that a nonzero polynomial $f \in \mathbb{N}_0[x^{\pm 1}]$ is \emph{monolithic} if $f = gh$ implies that either $g$ or $h$ is a monomial of $\mathbb{N}_0[x^{\pm 1}]$.
\end{definition}

Next we describe the relation between monolithic and irreducible polynomials in $\mathbb{N}_0[x^{\pm 1}]$.

\begin{lemma} \label{lemma: monolithic and gcd 1 implies irreducibility}
	Let $f = \sum_{i = 0}^{n} c_ix^{k_i} \in\mathbb{N}_0[x^{\pm 1}]$ such that $|\supp(f)| > 1$. Then $f$ is irreducible in $\mathbb{N}_0[x^{\pm 1}]$ if and only if $f$ is monolithic and $\gcd(c_0, \ldots, c_n) = 1$. 
\end{lemma}

\begin{proof}
	The direct implication follows from the fact that $\mathbb{N}_0[x^{\pm 1}]^{\times} = \{x^k \mid k \in \mathbb{Z}\}$. Suppose now that $f$ is monolithic and $\gcd(c_0, \ldots, c_n) = 1$. Let $g, h \in \mathbb{N}_0[x^{\pm 1}]$ such that $f = gh$. Since $f$ is monolithic, there is no loss in assuming that $g$ is a monomial. Thus we can write $g = cx^k$ for some $c \in \mathbb{N}$ and $k \in \mathbb{Z}$. Since $\gcd(c_0, \ldots, c_n) = 1$, we have that $c = 1$. Therefore the element $g$ is a unit of $\mathbb{N}_0[x^{\pm 1}]$ which, in turn, implies that $f$ is irreducible.
\end{proof}

Observe that, according to Lemma~\ref{lemma: monolithic and gcd 1 implies irreducibility}, a monolithic polynomial $f$ of degree $d$ is ``almost irreducible" in the sense that its only factors are either constants or polynomials of degree $d$. We now introduce a notion that is going to play an important role in the proof of Theorem~\ref{theorem: our main result}.

\begin{definition} \label{def: hyper-monolithic}
	A polynomial $f = \sum_{i = 0}^{n} c_ix^{k_i} \in\mathbb{N}_0[x^{\pm 1}]$ is called \emph{hyper-monolithic} provided that $|\supp(f)| > 1$ and either $k_0 - k_1 < k_i - k_{i + 1}$ for every $i \in \llbracket 1,n - 1 \rrbracket$ or $k_{n - 1} - k_n < k_j - k_{j + 1}$ for every $j \in \llbracket 0,n - 2 \rrbracket$.
\end{definition}

We now show that hyper-monolithic polynomials are, in fact, monolithic.

\begin{lemma} \label{lemma: monolithic if smallest gap occured at the final position}
	If a polynomial $f \in\mathbb{N}_0[x^{\pm 1}]$ is hyper-monolithic, then $f$ is monolithic.
\end{lemma}

\begin{proof}
	Write $f = \sum_{i = 0}^{n} c_ix^{k_i}$. We know that either $k_0 - k_1 < k_i - k_{i + 1}$ for every $i \in \llbracket 1,n - 1 \rrbracket$ or $k_{n - 1} - k_n < k_j - k_{j + 1}$ for every $j \in \llbracket 0,n - 2 \rrbracket$. Since the proofs of both cases are symmetric, we may assume that $k_{n - 1} - k_n < k_j - k_{j + 1}$ for every $j \in \llbracket 0,n - 2 \rrbracket$. Suppose towards a contradiction that $f = gh$, where neither $g$ nor $h$ is a monomial. Then write $g =\sum_{i=0}^md_ix^{t_i}$ and $h=\sum_{i=0}^{\ell}e_ix^{r_i}$ for positive integers $m, \ell \in \mathbb{N}$. Note that $k_n = t_m + r_{\ell}$. Without loss of generality, assume that the inequality $t_{m-1} -t_m \leq r_{\ell-1} - r_{\ell}$ holds. This implies that $k_{n - 1} = t_{m - 1} + r_{\ell}$. Thus,
	\[
	t_{m-1} - t_m = (t_{m-1} + r_{\ell}) - (t_m + r_{\ell}) = k_{n-1} - k_n.
	\] 
	Since $k_{n-1} \leq t_m +  r_{\ell-1} < t_{m-1} + r_{\ell-1}$, there exists an index $j\in [\![0, n-2]\!]$ such that $$k_j - k_{j+1} \leq t_{m-1} - t_m = k_{n-1} - k_n.$$ This contradiction proves that our hypothesis is untenable. Hence $f$ is monolithic.
\end{proof} 

Note that, given a hyper-monolithic polynomial $f \in \mathbb{N}_0[x^{\pm 1}]$, we can subtract as many as $f(1) - 2$ (multiplicative) units from $f$ without altering its hyper-monolithic nature, which justifies our use of the term ``hyper-monolithic" for these particular polynomials. Let us illustrate this with an example.

\begin{example} \label{ex: illustrating moving units}
	Consider the hyper-monolithic polynomial $f = 4x^7 + 3x^2 + x$. It is not hard to see that, for every $m, k \in \mathbb{N}_0$ with $0 \leq m \leq 4$ and $0 \leq k \leq 2$, we have that the polynomial $f' = (4 - m)x^7 + (3 - k)x^2 + x$ is hyper-monolithic. In other words, we can subtract up to six units from the polynomial $f$, and the remainder is still hyper-monolithic.
\end{example}

We now have gathered enough irreducibles and ``almost irreducibles" to take on Theorem~\ref{theorem: our main result}. Our attention turns to the task of writing a polynomial $f \in \mathbb{N}_0[x^{\pm 1}]$ as the sum of two irreducibles. Our approach involves initially writing $f$ as the sum of a polynomial $h$ and a hyper-monolithic polynomial $g$. Subsequently, we will shift some summands from $g$ to $h$, compelling $h$ to become irreducible. 

\begin{lemma}\label{lemma: splitting configuration}
	Let $\Delta_0, \ldots, \Delta_n$ be a sequence of positive integers with $n \geq 1$. There exist indices $\alpha, \beta \in \llbracket 0, n \rrbracket$ such that the following conditions hold:
	\begin{enumerate}
		\item $\alpha < \beta$;
		\item the inequality $\Delta_j > \max(\Delta_{\alpha}, \Delta_{\beta})$ holds for $j \in \llbracket 0, \alpha - 1 \rrbracket \cup \llbracket \beta + 1, n \rrbracket$;
		\item the inequality $\Delta_t < \max(\Delta_{\alpha}, \Delta_{\beta})$ holds for at most one index $t \in \llbracket \alpha, \beta \rrbracket$.
	\end{enumerate}
\end{lemma}

\begin{proof}
	We split our reasoning into the following two cases.
	
	\smallskip
	\noindent\textsc{Case 1:} $\min(\Delta_i)_{0 \leq i \leq n}$ occurs more than once in the sequence $\Delta_0, \ldots, \Delta_n$. In this case, let $\alpha \in \llbracket 0,n \rrbracket$ be the smallest index such that $\Delta_{\alpha} = \min(\Delta_i)_{0 \leq i \leq n}$, and let $\beta \in \llbracket 0,n \rrbracket$ be the largest index such that $\Delta_{\beta} = \min(\Delta_i)_{0 \leq i \leq n}$. It is easy to see that conditions $1$, $2$, and $3$ immediately follow.
	
	\smallskip
	\noindent\textsc{Case 2:} $\min(\Delta_i)_{0 \leq i \leq n}$ occurs once in the sequence $\Delta_0, \ldots, \Delta_n$. Let $\alpha \in \llbracket 0,n \rrbracket$ be the smallest index such that $\Delta_{\alpha}$ is either the smallest or the second smallest element of the sequence $\Delta_0, \ldots, \Delta_n$. Similarly, let $\beta \in \llbracket 0,n \rrbracket$ be the largest index such that $\Delta_{\beta}$ is either the smallest or the second smallest element of the same sequence. It is not hard to see that conditions $1$ and $2$ hold. Now observe that $\max(\Delta_{\alpha}, \Delta_{\beta})$ is precisely the second smallest element of the sequence $\Delta_0, \ldots, \Delta_n$. Consequently, there exists exactly one index $t \in \llbracket 0,n \rrbracket$ for which $\Delta_t <  \max(\Delta_{\alpha}, \Delta_{\beta})$, from which condition $3$ readily follows.
\end{proof}

We are rounding the final turn!

\begin{lemma} \label{lemma: f as a sum of a monolithic plus something with smaller size}
	Let $f\in\mathbb{N}_0[x^{\pm1}]$ satisfying that $|\supp(f)| \geq 3$. Then $f$ can be written as the sum of two polynomials $g, h \in \mathbb{N}_0[x^{\pm 1}]$ such that $g$ is hyper-monolithic and $h(1) \leq g(1)$.
\end{lemma}

\begin{proof}
	Write $f=\sum_{i=0}^nc_ix^{k_i}$. Consider the sequence of positive integers defined by $\Delta_0 \coloneqq k_0 - k_1, \ldots, \Delta_{n - 1} \coloneqq k_{n - 1} - k_n$, and let $\alpha,\beta\in \llbracket 0,n - 1 \rrbracket$ be indices such that conditions 1, 2, and 3 in the statement of Lemma~\ref{lemma: splitting configuration} hold (with respect to the sequence $\Delta_0, \ldots, \Delta_{n - 1}$). Additionally, let us consider the multiset $S = \{\Delta_i \mid i \in \llbracket 0, n - 1 \rrbracket\}$. Set 
	\[
	S_g \coloneqq \{\alpha, \alpha + 1\}  \cup \left(\llbracket \alpha + 2, \beta - 1 \rrbracket \cap \{\alpha + 2z + 1 \mid z \in \mathbb{N}\} \right) \cup \llbracket \beta + 2,n \rrbracket
	\]
	and 
	\[
	S_h \coloneqq \llbracket 0, \alpha - 1 \rrbracket \cup \left(\llbracket \alpha + 2, \beta - 1 \rrbracket \cap \{\alpha + 2z \mid z \in \mathbb{N}\}\right) \cup  \{\beta, \beta + 1\}.
	\]
	Note that $S_g$ and $S_h$ are not necessarily disjoint. Indeed, the index $\beta$ may be equal to $\alpha + 1$. Now if $\beta - \alpha$ is an odd number larger than $1$ and $\Delta_{\beta - 1} \leq \Delta_{\beta}$, then we remove $\beta + 1$ from $S_h$ and add it to $S_g$. Now consider the polynomial $g=\sum_{i=0}^{n}d_ix^{k_i}$ such that
	\begin{equation*}
		d_i =
		\begin{cases}
			c_i & i\in S_g,\\
			0 & \text{otherwise.}
		\end{cases}       
	\end{equation*}
	After removing the terms with coefficient $0$, we can rewrite $g=\sum_{i=0}^m e_ix^{t_i}$, where the coefficients $e_0,\ldots,e_m$ are positive integers and the exponents $t_0, \ldots, t_m$ satisfy $t_0> \cdots > t_m$. It is easy to see that $t_0 = k_{\alpha}$ and $t_1 = k_{\alpha+1}$. Moreover, for an arbitrary index $j\in \llbracket 1,m - 1 \rrbracket$, we have that $t_j - t_{j+1}$ is either the sum of at least two elements of $S$, one of which is greater than or equal to $\Delta_{\alpha}$, or is strictly greater than the difference $\Delta_{\alpha}$. Consequently, the inequality $t_0 - t_1 < t_j - t_{j+1}$ holds for every $j\in \llbracket 1,m - 1 \rrbracket$. Therefore the polynomial $g$ is hyper-monolithic. Similarly, consider the polynomial $h = \sum_{i = 0}^n a_ix^{k_i}$ such that 
	\begin{equation*}
		a_i =
		\begin{cases}
			c_i & i\in S_h,\\
			0 & \text{otherwise.}
		\end{cases}       
	\end{equation*}
	As before, we can rewrite $h = \sum_{i = 0}^{\ell} b_i x^{r_i}$, where the coefficients $b_0, \ldots, b_{\ell}$ are positive integers and the exponents $r_0, \ldots, r_{\ell}$ satisfy $r_0 > \cdots > r_{\ell}$. It is not hard to see that $r_{\ell-1} - r_{\ell} \leq \Delta_{\beta}$. Moreover, for an arbitrary index $j\in \llbracket 0,\ell-2 \rrbracket$, we have that $r_j - r_{j+1}$ is either the sum of at least two elements of $S$, one of which is greater than or equal to $\Delta_{\beta}$, or is strictly greater than the difference $\Delta_{\beta}$. Consequently, the inequality $r_{\ell - 1} - r_{\ell} < r_{j} - r_{j + 1}$ holds for every $j\in \llbracket 0,\ell-2 \rrbracket$. This implies that $h$ is also hyper-monolithic. Since $S_g \cup S_h = \llbracket 0,n \rrbracket$, we have that either $g(1) \geq \frac{f(1)}{2}$ or $h(1) \geq \frac{f(1)}{2}$ as, otherwise, the inequality $f(1) \leq g(1) + h(1) < f(1)$ would hold, but this is impossible. Consequently, we have that either $g(1) \geq (f - g)(1)$ or $h(1) \geq (f - h)(1)$, from which our argument concludes.
\end{proof}

We are now in a position to show that most polynomials $f \in \mathbb{N}_0[x^{\pm 1}]$ can be written as the sum of two irreducible polynomials.

\begin{lemma} \label{lemma: prime in the Nagura interval}
	Every $f\in \mathbb{N}_0[x^{\pm 1}]$ can be written as the sum of two irreducible polynomials provided that $|\supp(f)|\geq 3$ and $\frac{5f(1)}{6}-1\leq p\leq f(1)-2$ for some $p \in \mathbb{P}$. 
\end{lemma}

\begin{proof}
	By virtue of Lemma~\ref{lemma: f as a sum of a monolithic plus something with smaller size}, we can write $f = g + h$ for polynomials $g, h \in \mathbb{N}_0[x^{\pm 1}]$ such that $g$ is hyper-monolithic and $h(1) \leq g(1)$. Let us assume that $|\supp(g)| = 2$. Now write $g = ax^r + bx^s$, where $a, b \in \mathbb{N}$ and $r,s \in \mathbb{Z}$. Suppose without loss of generality that $a \geq b$. Thus,
	\[
	f(1) - h(1) - (b - 1) = a + 1 \geq \frac{g(1)}{2} + 1 \geq \frac{f(1)}{4} + 1,
	\]  
	which implies that
	\[
	h(1) + (b - 1) \leq \frac{3f(1)}{4} - 1 <  \frac{5f(1)}{6} - 1 \leq p.
	\]
	Note that then we can write
	\[
	f = \left[ h + (b - 1)x^s + (p - h(1) - b + 1)x^r \right] + \left[ (a - p + h(1) + b - 1)x^r + x^s \right],
	\]
	where the first summand between brackets is irreducible by Lemma~\ref{lemma: f(1) is prime} and the second summand between brackets is irreducible because $p \leq f(1) - 2$ (thus the coefficients of the second summand are positive). For the rest of the proof, we may assume that $|\supp(g)| \geq 3$. Write $g=\sum_{i=0}^nc_ix^{k_i}$, where the coefficients $c_0, \ldots, c_n$ are positive integers and the exponents $k_0, \ldots, k_n$ satisfy $k_0 > \cdots > k_n$. Take $j \in \llbracket 0,n \rrbracket$ such that $c_j = \min \{c_i \mid i \in \llbracket 0,n \rrbracket\}$. Since $|\supp(g)| \geq 3$, we have $c_j \leq \frac {g(1)}3$. Thus,
	\[
	h(1)+(c_j - 1)\leq\frac{f(1)}2+\frac{f(1)}3-1 = \frac{5f(1)}{6} - 1\leq p.
	\]
	Let $h' = h + (c_j - 1)x^{k_j}$, and let $g' = g - (c_j - 1)x^{k_j}$. It is easy to see that $f = h' + g'$, where $h'(1) \leq p$ and $g'$ is hyper-monolithic (as $g$ is hyper-monolithic) and irreducible by Lemma~\ref{lemma: monolithic and gcd 1 implies irreducibility} and Lemma~\ref{lemma: monolithic if smallest gap occured at the final position}. As we previously pointed out (see Example~\ref{ex: illustrating moving units} and the preceding comment), we can subtract up to $g'(1) - 2$ multiplicative units from $g'$ without affecting its hyper-monolithicness or irreducibility. Consequently, we can write $g' = g^* + g_*$, where $g^*(1) = p - h(1) - c_j + 1$ and $g_*$ is irreducible. It should be noted that such a representation is possible because $p \leq f(1) - 2$. Thus,
	\[
	f = \left[ h + (c_j - 1)x^{k_j} + g^* \right] + \left[ g_* \right],
	\] 
	where the first summand between brackets is irreducible by Lemma~\ref{lemma: f(1) is prime}.
\end{proof}

Bertrand's postulate states that, for all natural numbers $n \geq 1$, there exists a prime number $p \in \mathbb{P}$ such that $n < p < 2n$. Building upon this result, Nagura~\cite{nagura} showed that, for all natural numbers $n\geq 25$, there exists $p \in \mathbb{P}$ such that $n < p < \frac{6n}{5}$. Using this refinement of Bertrand's postulate along with the previous lemmas, we can now prove Theorem~\ref{theorem: our main result}. 

\begin{proof}[Proof of Theorem~\ref{theorem: our main result}]
	Let $f$ be a polynomial of $\mathbb{N}_0[x^{\pm 1}]$ satisfying that $f(1) > 3$ and $|\supp(f)| > 1$. If $|\supp(f)| = 2$, then we can write $f = ax^r + bx^s$, where $a, b \in \mathbb{N}$ and $r,s \in \mathbb{Z}$. There is no loss in assuming that $a \geq b$. In this case, it is not hard to argue that $f$ can be expressed as the sum of two irreducible polynomials. In fact, if $b = 1$, then $a \geq 3$ and, as such, $f$ can be expressed as the sum of $(a - 2)x^{r}+x^{s}$ and $2x^{r}$, which are both irreducibles. On the other hand, if $b > 1$, then $f$ can be expressed as the sum of the irreducibles $(a - 1)x^{r}+x^{s}$ and $x^{r}+(b - 1)x^{s}$.
	
	For the rest of the proof, we may assume that $|\supp(f)|\geq 3$. For now, let us also assume that $f(1) \geq 32$. Let $p$ be the largest prime number less than $f(1)$. Then we have $p \geq 31$, which implies that $p < f(1) < \frac{6p}{5}$. Thus $\frac{5f(1)}{6} < p < f(1)$. If $p \leq f(1) - 2$, then, by Lemma~\ref{lemma: prime in the Nagura interval}, our argument concludes. On the other hand, if $p = f(1) - 1$, then we can take $p'$ to be the largest prime less than $p$. Notice that $p' \geq 29$. As before, we have $p' < p <\frac{6p'}{5}$ which, in turn, implies that $\frac56(f(1) - 1) < p' < f(1) - 1$. Then our result follows from Lemma~\ref{lemma: prime in the Nagura interval}.

	To complete the proof, it remains to establish that our result applies to polynomials $f \in \mathbb{N}_0[x^{\pm 1}]$ satisfying $f(1) < 32$. We can assume that $f(1) \in \{11, 17, 23, 27, 29\}$. Indeed, since every number in the set $\llbracket 4,31 \rrbracket \setminus \{11, 17, 23, 27, 29\}$ can be expressed as the sum of two prime numbers, our result for these values follows from Corollary~\ref{cor: goldbach for numbers}. Now write $f=\sum_{i=0}^nc_ix^{k_i}$, where the coefficients $c_0, \ldots, c_n$ are positive integers. We split our argument into the following two cases.
	
	\smallskip
	\noindent\textsc{Case 1:} $f(1) \in \{11, 17, 23, 27\}$. In this case, we can write $f(1) = p + 4$ for some prime number $p \in \mathbb{P}$. If there exists some coefficient $c_j$ with $j \in \llbracket 0,n \rrbracket$ satisfying that $c_j \geq 3$, then we let $g = 3x^{k_j}+x^{k_t}$, where $t \in \llbracket 0,n \rrbracket \setminus \{j\}$. It is not hard to see that both $g$ and $f-g$ are irreducibles of $\mathbb{N}_0[x^{\pm 1}]$. Otherwise, for all $i\in \llbracket 0,n \rrbracket$, we have $c_i\leq 2$. Consequently, the inequality $n \geq 6$ holds. Then let $g = x^{k_{i_1}} + x^{k_{i_2}} + x^{k_{i_3}} + x^{k_{i_4}}$, where $k_{i_1} > k_{i_2} > k_{i_3} > k_{i_4}$ and $k_{i_1} - k_{i_2} \neq k_{i_3} - k_{i_4}$. As the reader can verify, the polynomial $g$ is either irreducible or the multiplication of two binomials, but the latter case is impossible as $k_{i_1} - k_{i_2} \neq k_{i_3} - k_{i_4}$. Hence $g$ and $f - g$ are both irreducibles, which concludes our argument for this case. 
	
	\smallskip
	\noindent\textsc{Case 2:} $f(1) = 29$. Notice that $f(1) = 23 + 6$. Consequently, if there exists some coefficient $c_j$ with $j \in \llbracket 0,n \rrbracket$ satisfying that $c_j\geq 5$, then we let $g = 5x^{k_j}+x^{k_t}$, where $t \in \llbracket 0,n \rrbracket \setminus \{j\}$. Since both $g$ and $f-g$ are irreducibles, our argument concludes. Otherwise, for all $i\in [\![0,n]\!]$, we have $c_i\leq 4$. By Lemma~\ref{lemma: f as a sum of a monolithic plus something with smaller size}, there exist polynomials $g, h \in \mathbb{N}_0[x^{\pm 1}]$ such that $f = g + h$, the polynomial $g$ is hyper-monolithic, and $h(1) \leq g(1)$. Let us write $g=\sum_{i=0}^rd_ix^{t_i}$, and let $m \in \llbracket 0,r \rrbracket$ such that $d_m = \min\{d_i \mid i \in \llbracket 0,r \rrbracket\}$. Clearly, the inequality $d_m \leq 4$ holds. Now write $g=g^*+(d_m-1)x^{t_m}$, and notice that $g^*$ is hyper-monolithic and irreducible by Lemma~\ref{lemma: monolithic and gcd 1 implies irreducibility} and Lemma~\ref{lemma: monolithic if smallest gap occured at the final position}. Moreover, we have that 
	\[
	g^*(1) = g(1) - (d_m - 1) = 29 - h(1) - (d_m - 1),
	\]
	while $h(1) + (d_m - 1) \leq 17$ as $h(1) \leq 14$ and $d_m \leq 4$. This implies that we can write $g^* = g' + g''$, where the polynomial $g'$ is hyper-monolithic and irreducible and $g''(1) = 17 - h(1) - (d_m -1)$. We therefore can write
	\[
	f=\left[h+g''+(d_m-1)x^{t_m}\right]+\left[g'\right],
	\]
	where each summand between brackets is irreducible. 
\end{proof}

With the horses back in the barn, we can now offer an example that illustrates how to use the results established in this paper to write a Laurent polynomial with positive integer coefficients as the sum of two irreducibles.

\begin{example}
	Consider the Laurent polynomial
	\[
	f=6x^7+3x^4+4x^3+7x+5x^{-1}+3x^{-4}+8x^{-6}\!
	\]  
	with positive integer coefficients. We let $S'$ denote the set consisting of the two smallest differences between consecutive exponents of $f$. Hence $S' = \{1, 2\}$. Observe that $3x^4 + 4x^3$ is the left-most binomial satisfying that the difference between its exponents is an element of $S'$; on the other hand, $3x^{-4} + 8x^{-6}$ is the right-most binomial satisfying the same property. Using these delimiters and following Lemma~\ref{lemma: f as a sum of a monolithic plus something with smaller size}, we can now split $f$ into two hyper-monolithic summands (see Definition~\ref{def: hyper-monolithic} above), namely
	\begin{equation*}
		g = 3x^4 + 4x^3 + 5x^{-1} \hspace{.3 cm}\text{ and } \hspace{.3 cm} h = 6x^7 + 7x + 3x^{-4} + 8x^{-6}.
	\end{equation*}
	Observe that the polynomial $h$ is irreducible by virtue of Lemma~\ref{lemma: monolithic and gcd 1 implies irreducibility} and Lemma~\ref{lemma: monolithic if smallest gap occured at the final position}. Moreover, subtracting $x^{-6}$ from $h$ does not affect its irreducibility (see Example~\ref{ex: illustrating moving units} and the preceding comment). Thus, we can write
	\[
	f = \left[ 3x^4 + 4x^3 + 5x^{-1} + x^{-6} \right] + \left[ 6x^7 + 7x + 3x^{-4} + 7x^{-6} \right]\!,
	\]
	where the first summand between brackets is also irreducible by Lemma~\ref{lemma: f(1) is prime}.
\end{example}

While the Goldbach conjecture is exceptionally challenging to verify, we have shown that an analogue for Laurent polynomials with positive integer coefficients is highly tractable. In fact, we have not only proven this result using elementary methods but have also provided an algorithm for writing a Laurent polynomial with positive integer coefficients as the sum of two irreducibles.

\bigskip
\section*{Acknowledgments}
The authors would like to thank anonymous referees and the editorial board of the American Mathematical Monthly for comments and suggestions that helped improve an early version of this paper. The authors would also like to extend their gratitude to the MIT PRIMES program for making this collaboration possible. While working on this paper, the second author was generously supported by the University of California President's Postdoctoral Fellowship.

\bigskip

\end{document}